\documentclass[a4paper]{article}
\usepackage{amsthm,amssymb,amsmath,enumerate,graphicx,epsf}

\usepackage{authblk}
\usepackage{color}
\usepackage{cite}

\newcommand{\COLORON}{1}
\newcommand{\NOTESON}{1}
\newcommand{\Debug}{0}

\newcommand{\comment}[1]{}
\newcommand{\COMMENT}[1]{}

\definecolor{darkgray}{rgb}{0.3,0.3,0.3}
\newcommand{\defi}[1]{{\color{darkgray}\emph{#1}}}

\comment{
	\begin{lemma}\label{}	
\end{lemma}
\begin{proof}

\end{proof}

\begin{theorem}\label{}
\end{theorem} 
\begin{proof} 	

\end{proof}

}

\newtheorem{proposition}{Proposition}[section]
\newtheorem{definition}[proposition]{Definition}
\newtheorem{theorem}[proposition]{Theorem}

\newtheorem{lemma}[proposition]{Lemma}

\newtheorem{conjecture}{{Conjecture}}[section]
\newtheorem{question}[conjecture]{Question}

\newtheorem{examp}[proposition]{Example}

\newtheorem*{remark}{Remark}

\newcommand{\FIG}{0}

\ifnum \NOTESON = 1 \newcommand{\note}[1]{ 

	\ 

	{\color{blue} \hspace*{-60pt} NOTE: \color{blue}{\small  \tt \begin{minipage}[c]{1.1\textwidth}  #1 \end{minipage} \ignorespacesafterend }} 
	
	\ 
	
	}
\else \newcommand{\note}[1]{} \fi

\newcommand{\afsubm}[1]{ \ifnum \Debug = 1 {\mymargin{#1}}
\fi} 

\ifnum \Debug = 1 
\else  \fi

\ifnum \FIG = 1 
\else  \fi

\ifnum \FIG = 1 
\else  \fi

\ifnum \Debug = 1 \usepackage[notref,notcite]{showkeys}
\fi

\ifnum \COLORON = 0 \renewcommand{\color}[1]{}
\fi

\newcommand{\N}{\ensuremath{\mathbb N}}

\newcommand{\cb}{\ensuremath{\mathcal B}}
\newcommand{\cc}{\ensuremath{\mathcal C}}

\newcommand{\cg}{\ensuremath{\mathcal G}}
\newcommand{\ch}{\ensuremath{\mathcal H}}

\newcommand{\cl}{\ensuremath{\mathcal L}}

\newcommand{\cs}{\ensuremath{\mathcal S}}

\newcommand{\eps}{\ensuremath{\epsilon}}

\newcommand{\isom}{\cong}

\newcommand{\seq}[1]{\ensuremath{(#1_n)_{n\in\N}}} 
\newcommand{\sseq}[2]{\ensuremath{(#1_i)_{i\in #2}}} 
 
\newcommand{\ceil}[1]{\ensuremath{\left\lceil #1 \right\rceil}}

\newcommand{\Aut}{\operatorname{Aut}}

\renewcommand{\Pr}{\mathbb{P}}

\newcommand{\Tr}[1]{Theorem~\ref{#1}}
\newcommand{\Sr}[1]{Section~\ref{#1}}

\newcommand{\Cnr}[1]{Con\-jecture~\ref{#1}}

\renewcommand{\iff}{if and only if}
\newcommand{\fe}{for every}

\newcommand{\st}{such that}

\newcommand{\wrt}{with respect to}

\newcommand{\labtequ}[2]{ \begin{equation} \label{#1} 	\begin{minipage}[c]{0.9\textwidth}  #2 \end{minipage} \ignorespacesafterend \end{equation} }

\newcommand{\mymargin}[1]{
  \marginpar{%
    \begin{minipage}{\marginparwidth}\small%
      \begin{flushleft}%
        {\color{blue}#1}%
      \end{flushleft}%
   \end{minipage}%
  }%
}%

\newcommand{\mySection}[2]{}

\newcommand{\Ge}{\mathcal{G}}
\newcommand{\Ce}{\mathcal{C}}
\newcommand{\Be}{\mathcal{B}}
\newcommand{\Le}{\mathcal{L}}

\title{Limits of subcritical random graphs and random graphs with excluded minors}

\date{}
\begin{document}

\author[1]{Agelos Georgakopoulos\thanks{Supported by EPSRC grant EP/L002787/1, and by the European Research Council (ERC) under the European Union’s Horizon 2020 research and innovation programme (grant agreement No 639046).}}
 
\author[2]{Stephan Wagner\thanks{Supported by the National Research Foundation of South Africa, grant number 96236.}}
\affil[1]{{Mathematics Institute}\\
 {University of Warwick}\\
  {CV4 7AL, UK}\\
}
\affil[2]{Department of Mathematical Sciences\\ Stellenbosch University\\ Private Bag X1, Matieland 7602, South Africa\\ \texttt{swagner@sun.ac.za}}
\maketitle

\begin{abstract}
We prove local convergence results for the uniformly random, labelled or unlabelled, graphs from subcritical families. As an example special case, we prove Benjamini-Schramm convergence for the uniform random unlabelled tree.

We introduce a compactification of the space of countable (connected) rooted graphs, and use it to generalise the notion of Benjamini-Schramm convergence in order to allow for vertices of infinite degree in the limit object.
\end{abstract}


\section{Introduction}

We prove convergence results for the uniform random graphs from subcritical families, and conjecture generalisations for minor-closed families.

Subcriticality is defined by a technical condition involving generating functions, which we recall in \Sr{secSubcrit} after an overview of known results. Important examples of subcritical  classes include cacti, outerplanar graphs and series-parallel graphs. These example classes can also be characterised via forbidden minors. A well-known conjecture of Noy \cite{Noy} states that an addable, minor-closed class is subcritical, if and only if it has a planar forbidden minor, but we disprove this conjecture in a follow-up paper \cite{AntiNoy}. Loosely speaking, subcritical families are thought to be `tree-like', and indeed we prove that their Benjamini-Schramm limits --- \defi{BS} for short, see Section~\ref{SecBS} for the definition--- are very similar to that of random trees. We now summarise our results.

\begin{theorem}[see also Theorem~\ref{thm:subcritical_labelled_almostsure}]\label{thmIntroLabel}
	Let $\Ce$ be a subcritical family of labelled connected graphs, and $R_n$ a uniformly random element of $\Ce$ with $n$ vertices. 
	Then the sequence $R_1,R_2,\ldots$ converges almost surely in the Benjamini-Schramm sense.
\end{theorem}

We remark that if we remove the phrase `almost surely' we obtain a weaker statement which we prove as a stepping stone  in Theorem~\ref{thm:subcritical_labelled}, which also provides a description of the limit object of Theorem~\ref{thmIntroLabel}.


When $\Ce$ is the class of trees, this weaker statement is a well-known fact that can be traced back to Grimmett \cite{GriRan}. (We do not know a reference for the almost sure convergence for trees.)

\medskip

With additional work, we also prove the analogous statements for unlabelled graphs:

\begin{theorem}[see also Theorem~\ref{thm:subcritical_unlabelled_almostsure}]\label{thmIntroUnlab}
	Let $\Ce$ be a subcritical family of unlabelled connected graphs, and $R_n$ a uniformly random element of $\Ce$ with $n$ vertices. 
	Then the sequence $R_1,R_2,\ldots$ converges almost surely in the Benjamini-Schramm sense.
\end{theorem} 

The limit random graph is described in Theorem~\ref{thm:subcritical_unlabelled_BS}.
A special case of this is that the uniform unlabelled tree BS-converges. This was independently proved by Stufler \cite{StuCon}.

We also prove the corresponding statement for unlabelled rooted graphs (Section~\ref{secSubcrit_unlabelled}). It turns out that the BS-limit for unrooted unlabelled graphs is not the same as the weak limit of rooted unlabelled graphs, see Theorems~\ref{thm:subcritical_unlabelled_rooted} and~\ref{thm:subcritical_unlabelled_BS}\footnote{As we mention below, the weak convergence of uniformly random labelled graphs and rooted unlabelled graphs from subcritical classes was independently obtained by Stufler \cite{StuRan}; the BS-convergence of (unrooted) unlabelled graphs is new.}.

\medskip
It is natural to conjecture that such convergence results hold for the uniform random graph on $n$ vertices from any minor-closed class. This is however not always true: let for example $\Ce$ be the class of `apex forests', i.e.\ the graphs $G$ such that for some vertex $x\in V(G)$, the subgraph $G - x$ is a forest. This is a minor-closed class: it can be characterised by forbidding the disjoint union of two triangles as a minor. Then a uniform random $n$-vertex graph from $\Ce$ has a similar distribution to a uniform random forest with $(n-1)$ vertices with an additional vertex $x$ which is adjacent to each vertex of the forest with probability $1/2$\footnote{In an earlier draft we made the false conjecture that the uniform random graph on $n$ vertices from any non-trivial minor-closed class BS-converges. We would like to thank Louigi Addario-Berry and Johannes Carmesin for pointing out counterexamples including this one.}.

The failure of BS-convergence in these examples is due to arbitrarily high degree vertices appearing in the neighbourhood of the root. Motivated by this we introduce, in Section~\ref{secInfDeg}, a notion of convergence for rooted (deterministic) graphs that allows infinite degree vertices in the limit graph. For example, the star with $n$ leaves will converge to a star with countably infinitely many leaves. This notion yields a compactification of the set of (isomorphism classes of) countable rooted graphs. Considering the weak topology on this compact space we obtain a notion of convergence for sequences of random graphs that generalises BS-convergence. Interestingly, compactness implies that every sequence has a convergent subsequence in this context, although this is not true with respect to BS-convergence. 

We conjecture that the uniformly random $n$-vertex graph from any minor-closed family converges in this sense.

\section{Benjamini-Schramm Convergence} \label{SecBS}

Given a sequence $\seq{R}$ of rooted graphs, we say that $(R_n)$ converges 
\defi{weakly}, if \fe\ $r\in \N$, the ball of radius $r$ in $R_n$ centred at its root converges in distribution as $n$ tends to infinity. This definition coincides with the usual notion of weak convergence when we endow the class of finite and locally finite graphs with the \defi{neighbourhood metric}; see \cite{BeSchrRec} or Section~\ref{secInfDeg}.

Let $\seq{G}$ be a sequence of (unrooted) finite graphs, where each $G_n$ might be a random graph with an arbitrary distribution. We can derive a random rooted graph $R_n$ from each $G_n$ by rooting $G_n$ at a vertex chosen uniformly at random from $V(G_n)$. We say that $(G_n)$ converges in the \defi{Benjamini-Schramm} sense (Aldous \& Steele \cite{AlStObj} call this the \defi{local weak} sense), or \defi{BS-converges} for short, if $(R_n)$ converges weakly.

\section{Labelled Subcritical Graph Classes} \label{secSubcrit}

In this and the following section, we prove the weak convergence of labelled subcritical graph families.  As mentioned earlier, important examples of subcritical graph classes include cacti, outerplanar graphs and series-parallel graphs.  Loosely speaking, subcritical families are ``tree-like'', and indeed their BS-limit is very similar to that of random trees. Subcritical graph classes have also been observed to exhibit tree-like behaviour of a different kind: specifically, Panagiotou, Stufler and Weller \cite{PanStuWel} showed that the scaling limit of random graphs from subcritical classes, in the Gromov-Hausdorff sense, is Aldous' continuum random tree. The degree distribution was studied, as already mentioned earlier, by Bernasconi, Panagiotou and Steger \cite{BePaStDeg}, while extremal parameters such as the maximum degree and the diameter are considered in \cite{drmota13}. More recently, and parallel to this work, Stufler \cite{StuRan} provided an approach to subcritical graphs via decorated trees, which can be used to obtain the scaling limit, but also local weak convergence of rooted graphs; his Theorems 1.11 and 1.13  parallel our Theorems~\ref{thm:subcritical_labelled} and~\ref{thm:subcritical_unlabelled_rooted}, although his approach is of a more probabilistic nature than ours, and he gives a different description of the limit objects. 

\subsection{Subcritical Details} \label{subdet}

This section contains the technical background for the results that follow. 
Let us first recall the formal definition of a subcritical graph class.

\begin{definition}
A class of graphs $\Ge$ is called \defi{block-stable} if it contains the graph $K_2$  and has the property that a graph $G$ belongs to $\Ge$ if and only if each of its blocks (maximal $2$-connected subgraphs) belongs to $\Ge$\footnote{A minor-closed class is block-stable if and only if it is addable; that is, each excluded minor is 2-connected.}. If $\Ge$ is such a block-stable class of labelled graphs, $\Ce$ the class of connected graphs in $\Ge$ and $\Be$ the class of blocks ($2$-connected graphs) in $\Ge$, then, following the notation of \cite{flajolet}, one has the symbolic decomposition
$$\Ge = \mathrm{Set}(\Ce),\ \Ce^{\bullet} = \mathcal{Z} \times \mathrm{Set}(\Be' \circ \Ce^{\bullet}),$$
where $\mathcal{Z}$ stands for a single vertex, $\Ce^{\bullet}$ is the class of graphs in $\Ce$ with a distinguished root, and $\Be'$ denotes the class derived from $\Be$ by not labelling one of the vertices. On the level of (exponential) generating functions $G(z)$, $C(z)$ and $B(z)$ associated with $\Ge$, $\Ce$ and $\Be$ respectively, this yields
$$G(z) = \exp(C(z))$$
and
$$C^{\bullet}(z) = z \exp\big(B'(C^{\bullet}(z)) \big),$$
where $C^{\bullet}(z) = z C'(z)$ is the exponential generating function for $\Ce^{\bullet}$. We call the class \defi{subcritical} if the radii of convergence $\rho$ and $\eta$ of $C(z)$ 
 and $B(y)$ satisfy the inequality $C^{\bullet}(\rho) < \eta$. 
\end{definition}

\begin{remark}
The requirement that $K_2$ be a possible block is not crucial and was mostly included for technical convenience in \cite{drmota11}. In the aforementioned examples, this condition is of course satisfied, as it is for every nontrivial minor-closed family of graphs.
\end{remark}

If these technical conditions are satisfied, then the generating function $C^{\bullet}(z)$ has a square root singularity at its radius of convergence $\rho$ \cite[Lemma 7]{drmota11}, i.e. 
\begin{equation}\label{eq:asymp_expansion}
C^{\bullet} = a - b (1-z/\rho)^{1/2} + O(|1-z/\rho|)
\end{equation}
around $\rho$ for suitable constants $a,b$, and $C^{\bullet}$ has no further singularities whose absolute value is $\rho$. Since $y = C^{\bullet}(z)$ is defined implicitly by the equation $y = z \exp(B'(y))$, we know that the partial derivative with respect to $y$ has to vanish at the singularity, i.e.
\begin{equation}\label{eq:sing_equation}
1 = \frac{d}{dy} z \exp \big(B'(y) \big) \Big|_{z=\rho,y=C^{\bullet}(\rho)},
\end{equation}
for otherwise there would be an analytic continuation of $C^{\bullet}$ around $\rho$ by the implicit function theorem (see e.g.\ \cite[Appendix B.5]{flajolet}). We will require this identity later. Singularity analysis (see Section VI of \cite{flajolet}, specifically Theorem VI.4) yields an asymptotic formula for the coefficients of $C^{\bullet}$ from the expansion around the singularity:
\begin{equation}\label{eq:asymp_number}
\frac{c_n}{n!} = [z^n] C^{\bullet}(z) \sim A \cdot n^{-3/2} \cdot \rho^{-n},
\end{equation}
where $c_n$ denotes the number of graphs of order $n$ in $\cc^{\bullet}$ and the constant $A$ is given by $A = b/(2\sqrt{\pi})$ ($b$ as in~\eqref{eq:asymp_expansion}).

To describe the BS-limit of subcritical families, we define the concept of a $2$-ended \defi{link}. A $2$-ended link is a graph in $\Ge$ that is obtained from a $2$-connected graph $B \in \Be$ with two distinct distinguished vertices, called the \defi{source} $s$ and the \defi{sink} $t$, by identifying each of the vertices of $B$ except for the sink with the root of some graph in $\cc^{\bullet}$. The latter rooted graphs are called the \defi{branches} of the link.

On the set $\Le$ of all unlabelled $2$-ended links, we define a probability measure $p_{\Le}$ in the following way: for $L \in \Le$, let $\ell(L)$ be the number of nonisomorphic labellings of $L$, where all vertices except for the sink are labelled, let $|L|$ be the number of labelled (i.e.\ non-sink) vertices of $L$, and set
\begin{equation}\label{pL}
p_{\Le}(L) = \frac{\ell(L) \rho^{|L|}}{|L|!}.
\end{equation}
To prove that this is in fact a probability measure, we need some properties of the generating function: recall once again that we have
$$C^{\bullet}(z) = z \exp\big(B'(C^{\bullet}(z)) \big).$$
In view of~\eqref{eq:sing_equation}, the dominant singularity $\rho$ of $C^{\bullet}$ must satisfy the equation
$$1 = \rho \exp \big(B'(C^{\bullet}(\rho)) \big) B''(C^{\bullet}(\rho)) = C^{\bullet}(\rho ) B''(C^{\bullet}(\rho)).$$
Now let us interpret the generating function $C^{\bullet}(z) B''(C^{\bullet}(z))$ combinatorially: $C^{\bullet}(z)$ is simply the generating function for rooted graphs in $\Ge$, and  $B''(C^{\bullet}(z))$ is the generating function for a derived block of which a second vertex has been distinguished (but not labelled), with a rooted graph in $\Ge$ substituted for every labelled vertex. This can precisely be interpreted as a labelled element of $\Le$: the factor $C^{\bullet}(z)$ represents the source branch, and $B''(C^{\bullet}(z))$ the rest of the link. Thus
$$\sum_{L \in \Le} p_{\Le}(L) = \sum_{L \in \Le} \frac{\ell(L) \rho^{|L|}}{|L|!} = C^{\bullet}(\rho ) B''(C^{\bullet}(\rho))  = \rho \exp\big(B'(C^{\bullet}(\rho))\big) B''(C^{\bullet}(\rho)) = 1.$$
\begin{remark}
The probability measure $p_{\Le}$ on $2$-ended links can also be described by a two-step procedure that determines the block containing source and sink first and the branches afterwards.
\end{remark}

\subsection{Convergence}
Now we are ready for the main theorem in the labelled setting:

\begin{theorem}\label{thm:subcritical_labelled}
Let $\Ce^{\bullet}$ be a subcritical family of labelled rooted connected graphs, and $R_n$ a uniformly random element of $\Ce^{\bullet}$ with $n$ vertices. Then $R_n$ converges weakly 
to a random infinite rooted graph that can be described as follows:
\begin{itemize}
\item The root is the source of a first $2$-ended link $L_1 \in \Le$, chosen randomly according to the probability measure $p_{\Le}$.
\item For every $j \in \N$, identify the sink of $L_j$ with the source of a $2$-ended link $L_{j+1} \in \Le$, again chosen randomly according to the probability measure $p_{\Le}$.
\item The result is an infinite chain of $2$-ended links $L_1,L_2,\ldots$. The choices are pairwise independent.
\end{itemize}
\end{theorem} 


In fact, we prove a slightly stronger statement than  weak convergence to the aforementioned limit object:

\begin{lemma}\label{lem:blockshape}
Let $\Ce^{\bullet}$ be a subcritical family of labelled rooted connected graphs.
Let $L_1,L_2,\ldots,L_k \in \Le$ be a finite sequence of $2$-ended links. Consider the event $\mathcal{S}(L_1,L_2,\ldots,L_k)$ that a uniformly random element of $\Ce^{\bullet}$ with $n$ vertices has a structure that can be described in the following way:
\begin{itemize}
\item the root is the source of $L_1$,
\item for $1 \leq j < k$, the sink of $L_j$ is identified with the source of $L_{j+1}$,
\item finally, the sink of $L_k$ is identified with the root of a rooted graph in $\Ce^{\bullet}$.
\end{itemize}
The probability that this event occurs tends to $\prod_{j=1}^k p_{\Le}(L_j)$ as $n \to \infty$.
\end{lemma}

\begin{proof}
Recall that the number $c_n$ of elements of order $n$ in $\Ce^{\bullet}$ is asymptotically given by
$$\frac{c_n}{n!} \sim A \cdot n^{-3/2} \cdot \rho^{-n}$$
as $n \to \infty$, cf.~\eqref{eq:asymp_number}. Now simply note that the number of elements of order $n$ in $\Ce^{\bullet}$ that have the structure that determines $\mathcal{S}(L_1,L_2,\ldots,L_k)$ is given by
\begin{multline*}
\binom{n}{|L_1|,|L_2|,\ldots,|L_k|,n-|L_1|-\cdots-|L_k|} \cdot \prod_{j=1}^k \ell(L_j) \cdot c_{n-|L_1|-\cdots-|L_k|} \\
= n! \cdot \prod_{j=1}^k (p_{\Le}(L_j)\rho^{-|L_j|}) \cdot \frac{c_{n-|L_1|-\cdots-|L_k|}}{(n-|L_1|-\cdots-|L_k|)!}.
\end{multline*}
In view of~\eqref{eq:asymp_number}, this is asymptotically equal to
\begin{align*}
n! \cdot \prod_{j=1}^k p_{\Le}(L_j) \cdot A \cdot (n-|L_1|-\cdots-|L_k|)^{-3/2} \cdot \rho^{-n} &\sim n! \cdot \prod_{j=1}^k p_{\Le}(L_j) \cdot A \cdot n^{-3/2} \cdot \rho^{-n} \\
& \sim c_n \prod_{j=1}^k p_{\Le}(L_j),
\end{align*}
which completes the proof. 
\end{proof}
Note that the events of the form  $\mathcal{S}(L_1,L_2,\ldots,L_k)$ are not actually pairwise disjoint, but for large enough $n$ they essentially are (in the sense that any two such events are disjoint for large enough $n$). Theorem~\ref{thm:subcritical_labelled} follows immediately by summing over all events $\mathcal{S}(L_1,L_2,\ldots,L_k)$ with $k > r$ that generate a given $r$-neighbourhood around the root.

As it turns out, an even stronger statement holds. In the following, which is a restatement of Theorem~\ref{thmIntroLabel}, we will prove that we have almost sure convergence in the Benjamini-Schramm sense if we take a sequence of random graphs from a subcritical class:

\begin{theorem}\label{thm:subcritical_labelled_almostsure}
Let $\Ce$ be a subcritical family of labelled connected graphs, and $R_n$ a uniformly random element of $\Ce$ with $n$ vertices. 
Then the sequence $R_1,R_2,\ldots$ converges almost surely in the Benjamini-Schramm sense to the random infinite graph described in Theorem~\ref{thm:subcritical_labelled}.
\end{theorem} 

Note the subtle difference to Theorem~\ref{thm:subcritical_labelled}: the graphs $R_n$ are unrooted in this context, and the sequence $R_1,R_2,\ldots$ is shown to converge \emph{deterministically} in the Benjamini-Schramm sense with probability $1$ once generated.

For the proof, we need a few definitions that will also become relevant in the following section. Let $H \in \cc'$ be a rooted graph with the property that its root belongs to only one block. Note that $H \in \cc'$ means that all vertices except for the root are labelled. We say that $H$ \defi{occurs at the fringe} of a graph $G \in \cc^{\bullet}$ if there exists a subgraph $K$ of $G$ that is isomorphic to $H$ (including the relative order of the labels) when an appropriate vertex $x$ of $K$ is chosen as a root, no vertex of $K$ other than $x$ has neighbours outside of $K$, and the root of $G$ does not lie in $K \setminus \{x\}$. Note that $x$ must be a cutvertex of $G$ in this case, unless $K = G$, and that $K \setminus \{x\}$ forms a connected component of $G \setminus \{x\}$.

In this scenario, $H$ is called a \defi{fringe subgraph} of $G$, and $K$ is called an \emph{occurrence} of $H$ as a fringe subgraph. We also define occurrences at the fringe of unrooted graphs: the condition is exactly the same, but it is (naturally) no longer required that the root of $G$ is not part of $K \setminus \{x\}$.

\begin{proof}[Proof of Theorem~\ref{thm:subcritical_labelled_almostsure}]
Instead of considering a sequence of random unrooted labelled graphs $R_n$, we rather first consider a sequence $R_n^{\bullet}$ of rooted labelled graphs whose roots we ``forget'' afterwards. Since every labelled graph with $n$ vertices corresponds to exactly $n$ rooted labelled graphs, this does not change the probability measure.

\begin{enumerate}
\item[Step 1.] Consider an arbitrary graph $H \in \cc'$ whose root only belongs to one block, and let $F_H(G)$ denote the number of occurrences of $H$ as a fringe subgraph in a rooted graph $G \in \cc^{\bullet}$. We define a bivariate generating function that takes $F_H(G)$ into account:
$$C_H^{\bullet}(z,u) = \sum_{G \in \cc^{\bullet}} \frac{z^{|G|}}{|G|!} u^{F_H(G)}.$$
This function satisfies the functional equation
\begin{equation}\label{eq:funct_eq_CH}
C_H^{\bullet}(z,u) = x \exp\Big(B'(C_H^{\bullet}(z,u)) + (u-1) \frac{z^{|H|}}{|H|!} \Big).
\end{equation}
The explanation is quite simple: the total number of occurrences of $H$ as a fringe subgraph is exactly equal to the total number of occurrences in branches rooted at vertices that lie in a common block with the root. Some of the root branches may be occurrences of $H$ as fringe subgraphs as well and thus add to the count. The term $(u-1) \frac{z^{|H|}}{|H|!}$ in the functional equation takes this into account.
\item[Step 2.] Now we analyse this functional equation: because of the subcriticality condition, it satisfies all conditions of \cite[Theorem 2.21]{Drmota2009Random}, which imply that $C_H^{\bullet}(z,u)$ (regarded as a function of $z$) has a dominant square root singularity for $u$ in a suitable neighbourhood of $1$:
$$C_H^{\bullet}(z,u) = g_H(z,u) + h_H(z,u) ( 1 - z/\rho_H(u))^{1/2}$$
for certain analytic functions $g_H$, $h_H$ and $\rho_H$, locally around $z = \rho = \rho_H(1)$ and $u=1$. As an immediate consequence (which follows by means of singularity analysis), we have, for some analytic function $\alpha_H$,
\begin{equation}\label{eq:asymp_CH}
[z^n] C_H^{\bullet}(z,u) = \alpha_H(u) n^{-3/2} \rho_H(u)^{-n} \big( 1 + O(n^{-1}) \big),
\end{equation}
uniformly in $u$ if $u$ is confined to a suitable neighbourhood of $1$.
\item[Step 3.] Note that
$$\frac{[z^n] C_H^{\bullet}(z,u)}{[z^n] C_H^{\bullet}(z,1)} = \frac{[z^n] C_H^{\bullet}(z,u)}{[z^n] C^{\bullet}(z)}$$
is exactly the probability generating function of $F_H(R_n^{\bullet})$, where $R_n^{\bullet}$ is a random rooted labelled graph in $\cc^{\bullet}$ with $n$ vertices. Equation~\eqref{eq:asymp_CH} shows that we are in the quasi-power regime, which implies a central limit theorem. More importantly, we obtain a strong concentration result for $F_H(R_n^{\bullet})$ including tail estimates: by \cite[Theorem 2.22]{Drmota2009Random}, we have
\begin{equation}\label{eq:tailbound}
\Pr \Big( | F_H(R_n^{\bullet}) - \mu_H n | \geq t \sqrt{n} \Big) \leq c_1 e^{-c_2 t^2}
\end{equation}
for certain constants $c_1,c_2,\mu_H$ (that depend on $H$). The constant $\mu_H$ is most relevant for us and can be calculated explicitly: it is given by $$\mu_H = \frac{\rho^{|H|}}{|H|! C^{\bullet}(\rho)}.$$
\item[Step 4.] If we take (for instance) $t = n^{1/4}$ in~\eqref{eq:tailbound} and apply the Borel-Cantelli lemma, then we find immediately that
\begin{equation}\label{eq:as_limit}
\lim_{n \to \infty} \frac{F_H(R_n^{\bullet})}{n} = \mu_H
\end{equation}
holds almost surely for every fixed $H \in \cc'$.
\item[Step 5.] Let us now discard the root of $R_n^{\bullet}$ to obtain an unrooted graph $R_n$. If $n > 2|H|+1$, then two occurrences of $H$ as a fringe subgraph cannot overlap. Hence the root of $R_n^{\bullet}$ is part of at most one occurrence of $H$ as a fringe subgraph of $R_n$, and it follows that~\eqref{eq:as_limit} still holds with $R_n^{\bullet}$ replaced by $R_n$. Let us now pick a new root $x$ for $R_n$, uniformly at random. The event $\mathcal{S}(L_1,L_2,\ldots,L_k)$ described in Lemma~\ref{lem:blockshape} can also be interpreted as follows: the two-ended links $L_1,L_2,\ldots,L_k$ form a fringe subgraph $H$ of which $x$ is a vertex (namely the source of $L_1$). The root of $H$ is the sink of $L_k$. Note that there are generally several possibilities for $H$ given $L_1,L_2,\ldots,L_k$, differing by their labels, and that for each choice of $H$ there may be several isomorphic choices of $x$, forming an orbit under the automorphism group of $H$. In view of~\eqref{eq:as_limit}, the probability that this situation occurs when the new root is chosen (considering $R_n$ as given) converges almost surely to a fixed constant, obtained by multiplying the constant $\mu_H$ from~\eqref{eq:as_limit} by the number of labellings and the size of the relevant orbit. In view of Lemma~\ref{lem:blockshape}, the constant must in fact be equal to $\prod_{j=1}^k p_{\cl}(L_j)$, and this can also be verified by a straightforward calculation. Since we have this almost sure convergence for all finite sequences $L_1,L_2,\ldots,L_k$, the theorem follows immediately.
\end{enumerate}
\end{proof}

\section{Unlabelled Subcritical Graph Classes}\label{secSubcrit_unlabelled}

\subsection{Weak Convergence of Rooted Graphs}

For rooted unlabelled graphs, the approach is essentially the same as in the labelled case. Let us first review the definition of subcriticality (and other important facts taken from \cite{drmota11}) in the unlabelled setting, which differs slightly. Unlike the labelled case, the distinction between rooted  and unrooted graphs is now significant.

First of all, we need the notion of \defi{cycle index series}. The cycle index associated with a graph $G$ is the cycle index of its automorphism group, i.e.~the polynomial in formal variables $s_1,s_2,\ldots$ given by
$$Z_G(s_1,s_2,\ldots) = \frac{1}{|\Aut(G)|} \sum_{\sigma \in \Aut(G)} s_1^{c_1(\sigma)}s_2^{c_2(\sigma)} \cdots,$$
where $c_k(\sigma)$ is the number of cycles of length $k$ in the cycle representation of $\sigma$ when viewed as a permutation of the set of vertices of $G$. Likewise, we associate a cycle index series to a family $\cg$ of graphs by summing the cycle indices of all its members:
$$Z_{\cg}(s_1,s_2,\ldots) = \sum_{G \in \cg} Z_G(s_1,s_2,\ldots).$$
The ordinary generating function of $\cg$ can be recovered by setting $s_1 = z$, $s_2 = z^2$, etc.

Once again, let us consider a block-stable class $\cg$, and let $\cc$ and $\cb$ denote the subclasses of connected graphs and blocks in $\cg$ respectively. We associate generating functions with these classes as in the previous section (one small difference being the fact that we are using ordinary generating functions rather than exponential generating functions now), as well as cycle index series. The different generating functions and cycle index series are connected by the following functional equations:
$$G(z) = \exp\Big( \sum_{i \geq 1} \frac{1}{i} C(z^i) \Big)$$
and
$$C^{\bullet}(z) = z \exp \Big( \sum_{i \geq 1} \frac{1}{i} Z_{\cb'}  (C^{\bullet}(z^i),C^{\bullet}(z^{2i}), \ldots) \Big),$$
where $\cc^{\bullet}$ stands for the class of connected graphs in $\cg$ with a distinguished root again, while $\cb'$ is the class of blocks in $\cg$ where one vertex is distinguished but not included in the count of cycles in the cycle index.
The latter equation is of the form
\begin{equation}\label{eq:fun_eq}
C^{\bullet}(z) = z \exp \Big( g(C^{\bullet}(z),z) + A(z) \Big),
\end{equation}
with
$$g(y,z) = Z_{\cb'}(y,C^{\bullet}(z^2),C^{\bullet}(z^3),\ldots)$$
and
$$A(z) = \sum_{i > 1} \frac{1}{i} Z_{\cb'}(C^{\bullet}(z^i),C^{\bullet}(z^{2i}),\ldots).$$
For later reference in the following subsection, we will also need a refinement of the generating function. The full cycle series $Z_{\cc^{\bullet}}$ associated with $\cc^{\bullet}$ satisfies the functional equation
\begin{equation}\label{eq:cycle_index_series}
Z_{\cc^{\bullet}}(s_1,s_2,\ldots) = s_1 \exp \Big( \sum_{i \geq 1} \frac{1}{i} Z_{\cb'}  \big(Z_{\cc^{\bullet}}(s_i,s_{2i},\ldots),Z_{\cc^{\bullet}}(s_{2i},s_{4i},\ldots), \ldots \big) \Big).
\end{equation}
The definition of subcriticality in the unlabelled setting hinges on properties of the functions in \eqref{eq:fun_eq}.

\begin{definition}
A class of unlabelled block-stable graphs is called subcritical if the functions in the functional equation~\eqref{eq:fun_eq} satisfy the following properties:
\begin{enumerate}
\item The radius of convergence $\rho$ of $C^{\bullet}(z)$ is non-zero,
\item $g(y,z)$ is analytic at the point $(C^{\bullet}(\rho),\rho)$,
\item the radius of convergence of $A(z)$ is greater than $\rho$,
\item the radius of convergence of the series $Z_{\cc}(0,z^2,z^3,\ldots)$ is greater than $\rho$.
\end{enumerate}
\end{definition}

Under these conditions, it follows again that the generating function $C^{\bullet}$ has a square root singularity at the radius of convergence $\rho$ that is dominant in the sense that there are no others with the same absolute value, i.e. the asymptotic expansion is of the form
\begin{equation}\label{eq:sing_exp_rho}
C^{\bullet}(z) = a - b (1-z/\rho)^{1/2} + O(|1-z/\rho|).
\end{equation}
As in the previous section, an asymptotic formula for the number of graphs of order $n$ in $\cc^{\bullet}$ follows from \eqref{eq:sing_exp_rho}. For the treatment of the unrooted case in the following section, we will also need a refinement taken from~\cite{drmota11}. Specifically, we consider the function $Z_{\cc^{\bullet}}(s,z^2,z^3,\ldots)$ derived from the cycle index series, in which only fixed points are marked with a special variable $s$. This function has a singular expansion of the form
\begin{equation}\label{eq:bivariate_expansion}
Z_{\cc^{\bullet}}(s,z^2,z^3,\ldots) = a(s,z) + b(s,z) (1-s/\rho(z))^{1/2},
\end{equation}
where $a$ and $b$ are analytic at $(s,z) = (\rho,\rho)$ (see \cite[Lemma 12]{drmota11}). This bivariate expansion becomes relevant in the step from unrooted to rooted graphs.

Let us now describe the limit object in the rooted unlabelled case. We define the family of $2$-ended links $\cl$ as in the previous section, but we have to modify the definition of the probability measure on $\cl$ somewhat. In the unlabelled setup, we define a probability measure $q_{\cl}$ by $q_{\cl}(L) = Z_L(\rho,\rho^2,\ldots)$. Here, the cycle index $Z_L$ associated with a link $L$ does not take the sink into account, which has to be kept fixed by automorphisms (as is also the case for the source). First of all, we have to verify that this is indeed a probability measure.

Note that the sum of all probabilities $q_{\cl}(L)$ is the generating function of $\cl$ evaluated at $z = \rho$. Thus we have
$$\sum_{L \in \cl} q_{\cl}(L) = \sum_{L \in \cl} Z_L(\rho,\rho^2,\ldots) = C^{\bullet}(z) Z_{\cb''}(C^{\bullet}(z),C^{\bullet}(z^2),\ldots) \Big|_{z =\rho}.$$
The first factor stands for the root branch attached to the source of the link, the second factor for the rest. Next recall that the function $y = C^{\bullet}(z)$ is determined by the implicit equation
$$y = z \exp( g(y,z) + A(z)),$$
with $g(y,z)$ and $A(z)$ as defined earlier. As in the labelled case (cf.~\eqref{eq:sing_equation}), the partial derivative with respect to $y$ must vanish at the singularity $\rho$, for otherwise $C^{\bullet}(z)$ could be continued analytically there by the implicit function theorem. Therefore, we have
$$1 = \frac{\partial}{\partial y} z \exp(g(y,z) + A(z))  \Big|_{z = \rho, y = C^{\bullet}(\rho)} = y \frac{\partial}{\partial y} g(y,z) \Big|_{z = \rho, y = C^{\bullet}(\rho)}.$$
Since 
$$\frac{\partial}{\partial y} g(y,z) = \frac{\partial}{\partial y} Z_{\cb'}(y,C^{\bullet}(z^2),C^{\bullet}(z^3),\ldots) = Z_{\cb''}(y,C^{\bullet}(z^2),C^{\bullet}(z^3),\ldots),$$
putting everything together gives
$$\sum_{L \in \cl} q_{\cl}(L) = C^{\bullet}(z) Z_{\cb''}(C^{\bullet}(z),C^{\bullet}(z^2),\ldots) \Big|_{z = \rho} = 1,$$
which is what we need. Now we are ready to prove weak convergence in the unlabelled case following the same method as in the previous section.

\begin{theorem}\label{thm:subcritical_unlabelled_rooted}
Let $\Ce$ be a subcritical family of unlabelled connected graphs, and $R_n$ a uniformly random element of $\Ce^{\bullet}$ with $n$ vertices. Then $R_n$ converges weakly to a limit object that can be described as follows:
\begin{itemize}
\item The root is the source of a first $2$-ended link $L_1 \in \Le$, chosen randomly according to the probability measure $q_{\Le}$.
\item For every $j \in \N$, identify the sink of $L_j$ with the source of a $2$-ended link $L_{j+1} \in \Le$, again chosen randomly according to the probability measure $q_{\Le}$.
\item The result is an infinite chain of $2$-ended links $L_1,L_2,\ldots$. The choices are pairwise independent.
\end{itemize}
\end{theorem}

Again, we obtain this by proving a slightly stronger statement:

\begin{lemma}\label{lem:blockshape_unlabelled}
Let $L_1,L_2,\ldots,L_k \in \Le$ be a finite sequence of $2$-ended links. Consider the event $\mathcal{S}(L_1,L_2,\ldots,L_k)$ that a rooted unlabelled graph with $n$ vertices in $\Ce^{\bullet}$ has a structure that can be described in the following way:
\begin{itemize}
\item the root is the source of $L_1$,
\item for $1 \leq j < k$, the sink of $L_j$ is identified with the source of $L_{j+1}$,
\item finally, the sink of $L_k$ is identified with the root of a rooted graph in $\Ce^{\bullet}$.
\end{itemize}
The probability that this event occurs tends to $\prod_{j=1}^k q_{\Le}(L_j)$ as $n \to \infty$.
\end{lemma}

\begin{proof}
We follow the lines of the proof of Lemma~\ref{lem:blockshape}. This time, we note that the cycle index series for rooted graphs satisfying the conditions of event $\mathcal{S}(L_1,L_2,\ldots,L_k)$, with the sink of the last $2$-ended link distinguished, is given by
$$Z_{L_1}(s_1,s_2,\ldots) \cdot Z_{L_2}(s_1,s_2,\ldots) \cdots Z_{L_k}(s_1,s_2,\ldots) \cdot Z_{\cc^{\bullet}}(s_1,s_2,\ldots).$$
Each of the factors $Z_{L_i}(s_1,s_2,\ldots)$ represents one of the two-ended links (note that the root and the distinguished sink of $L_k$ have to be fixed by any automorphism, implying that this is also the case for sources and sinks of all other links, which have to lie on any path from the source of $L_1$ to the sink of $L_k$).

The actual generating function is obtained by substituting $s_i = z^i$, as mentioned earlier, which yields
$$ Z_{L_1}(z,z^2,\ldots) Z_{L_2}(z,z^2,\ldots) \cdots Z_{L_k}(z,z^2,\ldots) C^{\bullet}(z).$$
Now recall that the generating function $C^{\bullet}(z)$ has a dominant square root singularity of the form
$$C^{\bullet}(z) = a - b (1-z/\rho)^{1/2} + O(|1-z/\rho|).$$
The prefactor $Z_{L_1}(z,z^2,\ldots) Z_{L_2}(z,z^2,\ldots) \cdots Z_{L_k}(z,z^2,\ldots)$, on the other hand, is simply a polynomial and therefore analytic at $\rho$. The square root singularity is therefore inherited from $C^{\bullet}(z)$, and the coefficient of $(1-z/\rho)^{1/2}$ only changes by a factor that is equal to the value of the aforementioned polynomial at $z = \rho$. Singularity analysis immediately yields that the proportion of graphs that are counted by our generating function tends to
$$Z_{L_1}(\rho,\rho^2,\ldots) Z_{L_2}(\rho,\rho^2,\ldots) \cdots Z_{L_k}(\rho,\rho^2,\ldots) = \prod_{j=1}^k q_{\Le}(L_j),$$
which is what we wanted to obtain. We observe that for large enough $n$, the sink of $L_k$ actually becomes unique (because it lies on the path from the root to the majority of vertices), so that it is no longer necessary to artificially keep it fixed. This completes our proof.
\end{proof}

\subsection{Benjamini-Schramm Convergence}

Now we finally consider the situation that we do not choose uniformly among all rooted connected graphs of $\cc$ but rather first randomly select a connected graph in $\cc$ (without root) and then choose one of the vertices as the root, uniformly at random. Note that this slightly changes the probability measure: several vertices may belong to the same orbit and may therefore lead to identical rooted graphs. The probability of such rooted graphs (with other vertices belonging to the root orbit) is therefore increased compared to the probability model of the previous section.

Bearing this argument in mind, it is not surprising that the Benjamini-Schramm limit is similar to the limit described in Theorem~\ref{thm:subcritical_unlabelled_rooted}, but not quite the same. The probability of a certain sequence $L_1,L_2,\ldots,L_k$ to occur changes, and in fact the different links can no longer be considered as independent random variables. One numerical example to illustrate this fact: the root of a random rooted unlabelled tree has limiting probability $0.338322$ to be a leaf, while a randomly chosen vertex of a random (rooted or unrooted) unlabelled tree has limiting probability $0.438156$ to be a leaf.

\begin{theorem}\label{thm:subcritical_unlabelled_BS}
Let $\Ce$ be a subcritical family of unlabelled connected graphs, and $R_n$ a uniformly random element of $\Ce$ with $n$ vertices. Choose one of the vertices of $R_n$ uniformly at random as the root to obtain $R_n^{\bullet}$. There exists a sequence of probability measures $P_{\Le}^{(j)}$ on $\Le^j$ such that
the random rooted graph $R_n^{\bullet}$ converges weakly to a limit object that can be described as follows:
\begin{itemize}
\item The root is the source of a first $2$-ended link $L_1 \in \Le$, chosen randomly according to the probability measure $P_{\Le}^{(1)}$.
\item For every $j \in \N$, identify the sink of $L_j$ with the source of a $2$-ended link $L_{j+1} \in \Le$, where the probability of $L_{j+1}$ to be chosen is the conditional probabilty
$$\frac{P_{\Le}^{(j+1)}(L_1,L_2,\ldots,L_{j+1})}{P_{\Le}^{(j)}(L_1,L_2,\ldots,L_{j})}.$$
\item The result is an infinite chain of $2$-ended links $L_1,L_2,\ldots$.
\end{itemize}
\end{theorem}

\begin{proof}
In order to obtain the correct Benjamini-Schramm limit of unrooted elements of $\cc$ from which a vertex is selected at random, we consider rooted elements first, but the root in this context is not actually the randomly chosen vertex where we want to find the local weak limit. We rather use a counting approach, where we determine how often a certain sequence of links occurs at the fringe of an element of $\cc^{\bullet}$, as we did in the proof of Theorem~\ref{thm:subcritical_labelled_almostsure}. Let us make this precise.

As in the labelled setup, consider a graph $H \in \cc^{\bullet}$ with the property that its root belongs to only one block. We say that $H$ occurs at the fringe of a graph $G \in \cc^{\bullet}$ 
if there exists a subgraph $K$ of $G$ that is isomorphic to $H$ when an appropriate vertex $x$ of $K$ is chosen as a root, no vertex of $K$ other than $x$ has neighbours outside of $K$, and the root of $G$ does not lie in $K \setminus \{x\}$. 

In this case, we call $H$ a fringe subgraph of $G$, and $K$ an occurrence of $H$. As before, we also define occurrences at the fringe of unrooted graphs by the same condition, without the restriction that the root may not be part of $K \setminus \{x\}$.

As a first step, we define a generating function for counting occurrences of a fixed graph $H$ at the fringe of a rooted graph. This count can be incorporated quite easily in the cycle index series: let us write $Z_{\cc^{\bullet},H}$ for the cycle index series of $\cc^{\bullet}$ with an additional variable $u$ whose exponent counts the number of occurrences of $H$. Equation~\eqref{eq:cycle_index_series} becomes, in analogy to~\eqref{eq:funct_eq_CH},
\begin{multline*}
Z_{\cc^{\bullet},H}(s_1,s_2,\ldots;u) \\
= s_1 \exp \Big( \sum_{i \geq 1} \frac{1}{i} Z_{\cb'}  \big(Z_{\cc^{\bullet}}(s_i,s_{2i},\ldots;u^i),Z_{\cc^{\bullet}}(s_{2i},s_{4i},\ldots;u^{2i}), \ldots \big) \\ 
+ \sum_{i \geq 1} \frac{1}{i} (u^i-1) Z_H(s_i,s_{2i},\ldots) \Big).
\end{multline*}
Again, this simply stems from the fact that the number of occurrences of $H$ in the different branches (that are joined at the root) are added, and the count is increased whenever one or more of the branches are isomorphic to $H$. We will only need a special case of the cycle index series, where the variable $s_1 = s$ is kept, while $s_i$ is replaced by $z^i$ for all $i > 1$. Let us denote this auxiliary function by
$R_H(s,z,u)$, and set $R(s,z) = R_H(s,z,1)$, which does not actually depend on $H$; note that the generating function of $\cc^{\bullet}$ is $C^{\bullet}(z) = R(z,z)$. The function $R_H(s,z,u)$ satisfies the functional equation
\begin{align*}
R_H(s,z,u) &= s \exp \Big( Z_{\cb'}  \big(R_H(s,z,u),R_H(z^{2},z^{2},u^{2}), \ldots \big) \\
&\quad + \sum_{i > 1} \frac{1}{i} Z_{\cb'}  \big(R_H(z^i,z^i,u^i),R_H(z^{2i},z^{2i},u^{2i}), \ldots \big) \\ 
&\quad + (u-1) Z_H(s,z^2,z^3,\ldots) + \sum_{i > 1} \frac{1}{i} (u^i-1) Z_H(z^i,z^{2i},\ldots) \Big).
\end{align*}
We are actually only interested in the total number of occurrences of $H$, which can be obtained from the partial derivative with respect to $u$, evaluated at $u=1$. Let us write $R_{H,u}(s,z,1)$ for this partial derivative and $Z_{\cb',k}$ and $Z_{H,k}$ for the derivative of cycle indices with respect to the $k$-th variable. Logarithmic differentiation gives us
\begin{align*}
\frac{R_{H,u}(s,z,1)}{R_H(s,z,1)} &= R_{H,u}(s,z,1) Z_{\cb',1}  \big(R_H(s,z,1),R_H(z^{2},z^{2},1),\ldots \big) \\
&\quad +  \sum_{k > 1} k R_{H,u}(z^k,z^k,1) Z_{\cb',k}  \big(R_H(s,z,1),R_H(z^{2},z^{2},1),\ldots \big) \\
&\quad + \sum_{i > 1} \sum_{k \geq 1}  k R_{H,u}(z^{ki},z^{ki},1) Z_{\cb',k}  \big(R_H(z^i,z^i,1),R_H(z^{2i},z^{2i},1),\ldots \big) \\
&\quad + Z_H(s,z^2,z^3,\ldots) + \sum_{i > 1} Z_H(z^i,z^{2i},\ldots).
\end{align*}
We compare this to the partial derivatives of $R_H(s,z,1) = R(s,z)$ with respect to $s$ and $z$ to show that they have the same type of singularity, and in the same position: we have
$$\frac{R_{s}(s,z)}{R(s,z)} = \frac{1}{s} + R_{s}(s,z) Z_{\cb',1}  \big(R(s,z),R(z^{2},z^{2}),\ldots \big),$$
and
\begin{align*}
\frac{R_{z}(s,z)}{R(s,z)} &= R_{z}(s,z) Z_{\cb',1}  \big(R(s,z),R(z^{2},z^{2}),\ldots \big) \\
&\quad + \sum_{k > 1} kz^{k-1} (R_{s}(z^k,z^k) + R_{z}(z^k,z^k)) Z_{\cb',k} \big(R(s,z),R(z^{2},z^{2}),\ldots \big) \\
&\quad + {\sum_{i > 1} \sum_{k \geq 1}} kz^{ki-1} (R_{s}(z^{ki},z^{ki},1) + R_{z}(z^{ki},z^{ki},1)) \\
&\qquad \cdot Z_{\cb',k} \big(R(z^i,z^i),R(z^{2i},z^{2i}),\ldots \big).
\end{align*}
Solving the equations, we find that
\begin{align}
R_{H,u}(s,z,1) &= \frac{A^{(u)}(s,z)R(s,z)}{1 - R(s,z) Z_{\cb',1}  \big(R(s,z),R(z^{2},z^{2}),\ldots \big)}, \nonumber \\
R_{H,s}(s,z,1) &= R_s(s,z) = \frac{A^{(s)}(s,z)R(s,z)}{1 - R(s,z) Z_{\cb',1}  \big(R(s,z),R(z^{2},z^{2}),\ldots \big)}, \label{eq:partial_deriv} \\
R_{H,z}(s,z,1) &= R_z(s,z) = \frac{A^{(z)}(s,z)R(s,z)}{1 - R(s,z) Z_{\cb',1}  \big(R(s,z),R(z^{2},z^{2}),\ldots \big)}, \nonumber
\end{align}
for certain functions $A^{(u)},A^{(s)},A^{(z)}$ that are still analytic around the singularity $(s,z) = (\rho,\rho)$. Note that $sR_{s}(s,z) + zR_{z}(s,z)$ is simply the generating function for $\cc^{\bullet}$ where each graph gets an additional weight equal to the number of vertices. Now we need the fact (established in \cite{drmota11}, see~\eqref{eq:bivariate_expansion} earlier) that $R(s,z)$ has a singular expansion of the form
$$a(s,z) + b(s,z) (1-s/\rho(z))^{1/2},$$
so it follows that $R_{H,u}(s,z,1),R_{s}(s,z),R_{z}(s,z)$ all have singular expansions of the form
$$a(s,z) + b(s,z) (1-s/\rho(z))^{-1/2}.$$
In order to take the step from rooted graphs to unrooted graphs, we use the fact that the cycle index series $Z_{\ch}$ of a family $\ch$ of graphs and the cycle index series of the derived class $\ch'$ are related by $Z_{\ch'}(s_1,s_2,\ldots) = \frac{\partial}{\partial s_1} Z_{\ch}(s_1,s_2,\ldots)$. Therefore, if $Q(s,z) = Z_{\cc}(s,z^2,z^3,\ldots)$ is the (reduced) cycle index series for the class of unrooted connected graphs in $\cg$, we have
\begin{equation}\label{eq:integrated}
Q(s,z) = Q(0,z) + \int_0^s \frac{R(w,z)}{w}\,dw
\end{equation}
and
$$s Q_s(s,z) + z Q_z(s,z) = z Q_z(0,z) + \int_0^s \Big(R_{s}(w,z) + \frac{zR_{z}(w,z)}{w} \Big)\,dw,$$
the latter stemming from the aforementioned fact that the operator $s \frac{\partial}{\partial s} + z \frac{\partial}{\partial z}$ simply induces a factor $n$ for every graph of order $n$. Now we want an analogous generating function taking occurrences of $H$ as fringe subgraphs into account. Recall that $H$ occurs on the fringe in an unrooted graph if this is the case for some choice of root. If we choose a root in a given unrooted graph, the number of occurrences of $H$ on the fringe in the rooted version may not be the same as in the unrooted graph and may also depend on the choice of root. This is due to the fact that an occurrence of $H$ is not counted in the rooted graph if the root is part of it. However, the influence of this effect is bounded (cf.~Step 5 in the proof of Theorem~\ref{thm:subcritical_labelled_almostsure}). Thus the integral
$$Q_H(s,z) = \int_0^s \frac{R_{H,u}(w,z,1)}{w}\,dw$$
analogous to~\eqref{eq:integrated} is a generating function in which each graph of $\cc$ is weighted with the number of occurrences of $H$ up to $O(1)$ (the term analogous to $Q(0,z)$ in~\eqref{eq:integrated} even contributes less by the assumptions on a subcritical class). Note, however, that the coefficients of $Q_H(s,z)$ are not actually precise counts and generally not even integers.

Since the function $Q_H(s,z)$ has the same kind of singular expansion as the generating function $sQ_s(s,z) + zQ_z(s,z)$, in which graphs are simply weighted by the number of vertices, we can conclude by singularity analysis that the average number of occurrences of a fixed rooted graph $H$ on the fringe of a randomly chosen unrooted graph $G$ of order $n$ in $\cc$ is $\mu_H n + O(1)$ for some constant $\mu_H$ that only depends on $H$ and the specific subcritical class. It can be written as
$$\mu_H = \frac{A^{(u)}(\rho,\rho)}{\rho(A^{(s)}(\rho,\rho)+A^{(z)}(\rho,\rho))},$$
with $A^{(u)},A^{(s)},A^{(z)}$ as in~\eqref{eq:partial_deriv}. The same is true, with the same constant $\mu_H$, for random rooted graphs.

Now we put the count of subgraphs on the fringe into our context, as in the proof of Theorem~\ref{thm:subcritical_labelled_almostsure}: we say that a vertex $w$ in a graph $G$ has a $k$-block neighbourhood given by a chain of links $L_1,L_2,\ldots,L_k$, if there exists a cutvertex $v$ in $G$ such that the component of $G \setminus v$ that contains $w$, together with the vertex $v$, forms a subgraph $K$ of $G$ isomorphic to the chain of links $L_1,L_2,\ldots,L_k$, with $w$ as the source of $L_1$ and $v$ as the sink of $L_k$. If the graph $G$ is rooted and its root does not belong to $K$, then we can regard $K$ as a fringe subgraph of $G$. Conversely, if $H$ is the graph formed by joining links $L_1,L_2,\ldots,L_k$, rooted at the sink of $L_k$, then every occurrence of $H$ as a fringe subgraph means that the vertex corresponding to the source of $L_1$ (or any vertex lying in the same orbit of $H$) has a $k$-block neighbourhood that is exactly given by the chain of links $L_1,L_2,\ldots,L_k$.

So to count the number of vertices with a specified $k$-block neighbourhood around them, we can count the number of occurrences of a certain fringe subgraph (weighted by the size of the respective orbit). Consequently, the mean number of occurrences of $H$ in a random graph $R_n \in \cc$ of order $n$, multiplied by the size of the source's orbit and divided by $n$, provides us with the probability that a randomly selected vertex of $R_n$ has a $k$-block neighbourhood given by the links $L_1,L_2,\ldots,L_k$. From the considerations above, we know that the mean number of occurrences of any given fringe subgraph is linear in $n$, which means that this probability tends to a constant. If we can show that the sum of all constants obtained in this way, summed over all choices of $L_1,L_2,\ldots,L_k$, is actually $1$, then these constants provide us with the desired limiting probabilities $P_{\Le}^{(k)}(L_1,L_2,\ldots,L_k)$ (we need tightness, i.e.~to guarantee that no probability mass ``escapes to infinity'').

To this end, consider rooted graphs again. We define the level of a vertex $v$ as the number of blocks that are traversed in a path from the root to $v$ (e.g., the root's level is $0$, all other vertices in blocks that contain the root belong to level $1$, etc.). For every vertex $v$ whose level is at least $k$, there exists a fringe subgraph that contains $v$ and that forms a $k$-block neighbourhood for $v$ (it consists of the first $k$ blocks traversed on a path from $v$ to the root and all vertices for which a path from the root has to use an edge in at least one of these $k$ blocks). Now we take the sum of all generating functions $R_{H,u}(z,z,1)$, each weighted by the number $\omega(H)$ of possible sources that turn $H$, with the sink at its root, into a chain of $k$ links, and call this sum $S^{(k)}(z)$. Each pair of a graph in $\cc^{\bullet}$ and a vertex whose level is at least $k$ gets counted exactly once by $S^{(k)}(z)$, since exactly one fringe subgraph $H$ (that represents the $k$-block neighbourhood around the vertex) corresponds to it. Thus we have
\begin{align*}
\sum_{L_1,L_2,\ldots,L_k} &P_{\Le}^{(k)}(L_1,L_2,\ldots,L_k) \\
&= \sum_{H} \omega(H) \mu_H = \sum_{H} \omega(H) \Big( \lim_{n \to \infty} \frac{[z^n] R_{H,u}(z,z,1)}{n [z^n] C^{\bullet}(z)}\Big) \\
&= \lim_{n \to \infty} \frac{[z^n] S^{(k)}(z)}{n [z^n] C^{\bullet}(z)} = 1 - \lim_{n \to \infty} \frac{n [z^n] C^{\bullet}(z) - [z^n] S^{(k)}(z)}{n [z^n] C^{\bullet}(z)}.
\end{align*}
So it remains to show that $n [z^n] C^{\bullet}(z) - [z^n] S^{(k)}(z) = O([z^n] C^{\bullet}(z))$. Note that $n [z^n] C^{\bullet}(z) - [z^n] S^{(k)}(z)$ counts graphs in $\cc^{\bullet}$ weighted by the number of vertices whose level is less than $k$ (those not counted by $S^{(k)}(z)$ according to the aforementioned argument). We write $T^{(k)}(z)$ for the associated generating function, which can be expressed as $z \frac{\partial}{\partial z} C^{\bullet}(z) - S^{(k)}(z)$; we will prove that it has a singularity of the same order as $C^{\bullet}(z)$ at $z = \rho$. The required estimate then follows immediately.

Let us define another auxiliary generating function $U^{(k)}(z,u)$ for graphs in $\cc^{\bullet}$, where the exponent of $u$ indicates the number of vertices whose level is less than $k$. Note that
$$C^{\bullet}(z) = R(z,z) = U^{(k)}(z,1)\quad \text{and}\quad T^{(k)}(z) = \frac{\partial}{\partial u} U^{(k)}(z,u) \Big|_{u=1}$$
by definition. Clearly, $U^{(1)}(z,u) = u C^{\bullet}(z)$ is just the ordinary generating function for $\cc^{\bullet}$ with an additional factor $u$ that stands for the root (the only level $0$ vertex). Moreover, the recursive decomposition of $\cc^{\bullet}$ gives us
$$U^{(k+1)}(z,u) = uz \exp \Big( \sum_{i \geq 1} \frac{1}{i} Z_{\cb'}  (U^{(k)}(z^i,u^i),U^{(k)}(z^{2i},u^{2i}), \ldots) \Big).$$
 Consequently,
\begin{align*}
T^{(k+1)}(z) &= \frac{\partial}{\partial u} U^{(k+1)}(z,u) \Big|_{u=1} \\
&= U^{(k+1)}(z,1) \Big(1 + \sum_{i \geq 1} \sum_{\ell \geq 1} \ell Z_{\cb',\ell}  (U^{(k)}(z^i,1),U^{(k)}(z^{2i},1), \ldots) \frac{\partial}{\partial u} U^{(k)}(z^i,u) \Big|_{u=1} \Big) \\
&= C^{\bullet}(z) \Big( 1 + \sum_{i \geq 1} \sum_{\ell \geq 1} \ell Z_{\cb',\ell}  (C^{\bullet}(z^i),C^{\bullet}(z^{2i}), \ldots) T^{(k)}(z^i) \Big).
\end{align*}
It now follows from a straightforward induction argument  that $T^{(k)}(z)$ has a square root singularity of the same kind as $C^{\bullet}(z)$ (namely~\eqref{eq:sing_exp_rho}), which implies that the coefficients of $T^{(k)}$ are bounded by a multiple of those of $C^{\bullet}$ for every fixed $k$. This completes our proof.
\end{proof}

In analogy to Theorem~\ref{thm:subcritical_labelled_almostsure}, we have even stronger (almost sure) convergence as mentioned in the introduction. We omit the technical details.

\begin{theorem}\label{thm:subcritical_unlabelled_almostsure}
Let $\Ce$ be a subcritical family of unlabelled connected graphs, and $R_n$ a uniformly random element of $\Ce$ with $n$ vertices. 
Then the sequence $R_1,R_2,\ldots$ converges almost surely in the Benjamini-Schramm sense to the random infinite graph described in Theorem~\ref{thm:subcritical_unlabelled_BS}.
\end{theorem} 

\section{Generalisation of BS-convergence with infinite degrees} \label{secInfDeg}
\newcommand{\CG}{\mathbb{G}}
\newcommand{\CGc}{\mathbb{G}^*}

In this section we define a pseudometric $d$ on the set $\CG$ of (isomorphism classes of) countable rooted graphs, and we prove that the quotient space of  $\CG$  with respect to the `$d=0$' relation is compact. By concidering the weak topology of probability measures on that space we obtain a generalisation of BS-convergence, the limit objects of which can be graphs with vertices of infinite degree.

\subsection{Preliminaries}
We will write $V(G), E(G)$ for the vertex set and edge set of a graph $G$ respectively.

If $(G,o)$ is a rooted graph, then a \defi{rooted subgraph} of $(G,o)$ is a rooted graph $(R,o)$ such that $R$ is a subgraph of $G$ (containing $o$). An \defi{induced} subgraph $H$ of $G$ is a subgraph that contains all edges $xy$ of $G$ with $x,y\in V(H)$. Note that an induced subgraph is uniquely determined by its vertex set. The subgraph of $G$ \defi{spanned} by a vertex set $S\subseteq V(G)$ is the induced subgraph of $G$ with vertex set $S$.

We will say that two rooted graphs $(G,o)$
and $(H,r)$ are \defi{isomorphic}, if there is a map $f: V(G) \to V(H)$ \st\ $f(o)=r$ and $f(x)f(y)\in E(H)$  \iff\ $xy \in E(G)$. Such an $f$ will be called an \defi{isomorphism}.

The \defi{neighbourhood metric} $d_N$ on the space of isomorphism classes of rooted locally finite graphs is defined as follows. Let
\begin{align}
	r_N(G,H):=  \sup \{r \mid & B_r(G) \isom B_r(H) \},
\end{align}
where $B_r$ denotes the ball of radius $r$ around the root, and $\isom$ is the isomorphism relation. Let $d_N(G,H):= 1/r_N(G,H)$. Recall that BS-convergence can be defined using the weak convergence of probability measures with respect to $d_N$ (\Sr{SecBS}).

\subsection{A compact pseudometric on the space of countable rooted graphs}

We proceed with the definition of the aforementioned pseudometric $d$.

A \defi{rooted connected induced subgraph}, or \defi{RCIS} for short, of a rooted graph $(G,o)$, is a connected, induced subgraph of $G$ that is also rooted at $o$ (in particular, it must contain $o$).

\begin{definition} \label{defd}
Given two rooted graphs $G,H$, we define their \defi{radius of similarity} $r(G,H)$ by 
\begin{align*} \label{rGH}
	r(G,H):=  \sup \{r \mid & \text{  every rooted graph on $r$ vertices is isomorphic to a }\\ & \text{ RCIS of $G$ \iff\ it is isomorphic to a RCIS of $H$} \}.
\end{align*}


We define our pseudometric $d$ by letting $d(G,H):= 1/r(G,H)$. 
\end{definition}

Note that we may have $r(G,H) =\infty$, in which case we have $d(G,H):=0$. This is the case when $G,H$ are isomorphic, but it can also happen for non-isomorphic infinite graphs; we provide such examples later on. The choice of the function $1/r(G,H)$ is not important; we can instead choose any decreasing function $f:\N \cup \{\infty\} \to [0,\infty)$ such that $\lim_{n\to \infty} f(n) =0 $ and $f(\infty) = 0$ and let $d(G,H):= f(r(G,H))$ to induce an equivalent topology on our space. 

It is straightforward to check that $d$ satisfies the triangle inequality; even more, similarly to the neighbourhood metric for locally finite graphs, it is an ultra-metric on the set of equivalence classes \wrt\ the relation $d(G,H)=0$. Indeed, this follows immediately from the fact that having radius of similarity $r$ is an equivalence relation for any $r\in\N$ by the definition of $r(G,H)$.

It is a good exercise to check that any sequence \seq{G} of rooted graphs that converges \wrt\ the neighbourhood metric also converges \wrt\ $d$; to see this, given $r\in\N$, choose $n_0$ large enough that the $r$-balls of all $G_n, n>n_0$, are isomorphic, and note that any $r$-vertex RCIS of $G_n$ is contained in the $r$-ball of $G_n$ by the definitions. 

Note also that the star on $n$ vertices converges to an infinite star. 

Let $\CG$ denote the set of isomorphism types of countable connected rooted graphs. 
Let $\CGc$ denote the quotient of $\CG$ with respect to the equivalence relation given by $d(G,H)=0$. We will prove that the topology induced on $\CGc$ by $d$ is compact, by showing that $(\CGc,d)$ is complete and totally bounded. 

Before doing so, we provide some examples of non-isomorphic graphs $G_i,H_i$ with $d(G_i,H_i)=0$.
\begin{itemize}
\item Let $G_0$ be the Rado graph, and let $H_0$ be the join of two copies of $G_0$. (The \defi{join} of two or more rooted graphs is the graph obtained from their disjoint union by identifying their roots into a single root vertex.)
\item Let $G_1$ be the join of the paths of length $n$ \fe\ $n\in\N$, and let $H_1$ be the join of $G_1$ with one or more infinite paths.
\item Let $G_2$ be the join of the fans of order $n$ \fe\ $n\in\N$, and let $H_2$ be the (one-way or two-way) infinite fan. (A \defi{fan} is a path together with an additional root vertex that is adjacent to each vertex of the path.)
\end{itemize}

Examples of countable graphs that are only at distance $d=0$ to themselves include all finite and locally finite graphs, the complete ($k$-partite) graph, and the infinite star.

We proceed with our proof of compactness.

\begin{lemma}\label{}	
	$(\CGc,d)$ is complete.
\end{lemma}
\begin{proof}
	Let $\seq{G}$ be a sequence of countable connected rooted graphs such that the corresponding equivalence classes of $\CGc$ form a Cauchy sequence in $(\CGc,d)$. To show completeness, we need to construct a limit graph $L$ such that $\lim_{n \to \infty} d(G_n,L) = 0$. The reader may choose to think of each $G_n$ as a finite graph, in which case $G_n$ is the only element of its equivalence class; although we do need to consider infinite $G_n$ to prove the completeness of $\CGc$, the cardinalities of the $G_n$ will make no difference for our arguments. 
	
	We will construct $L$ as a union of a sequence $\seq{L}$ of finite graphs with a common root such that $L_n \subset L_{n+1}$ \fe\ $n$.
	
	To begin with, fix an arbitrary integer $k > 2$, and let $K$ denote the set of isomorphism types of rooted connected graphs on $k$ vertices. Note that $|K|\leq 2^{k \choose 2 }$ is finite. We claim that 
	\labtequ{balls}{the number of distinct (clopen) balls of $(\CGc,d)$ of radius $1/k$ is at most $2^{|K|}$.}
	Indeed, by the definition of $d$, two graphs are at distance at most $1/k$ \iff\ they contain the same 
	RCIS's of order at most $k$, and so the number of balls of $(\CGc,d)$ of radius $1/k$ is at most the cardinality $2^{|K|}$ of the power set of $K$.
	
	Now since our sequence $\seq{G}$ is Cauchy, almost all of its members lie in one of these balls, in other words, almost all $G_i$ contain the same elements of $K$ as RCIS. Let $K'$ be the subset of $K$ comprising those graphs that appear in  almost all $G_i$, and let $i_k\in \N$ be large enough that $G_i$ contains all elements of $K'$ as RCIS for every $i>i_k$. For every $i>i_k$, let $X_i$ be a minimal RCIS of $G_i$ containing all  elements of $K'$ as RCIS. To see that $X_i$ always exists, note that a candidate can be obtained by taking the union of one RCIS of $G_i$ isomorphic to each element of $K'$. Since we are choosing $X_i$ to be minimal, its order is bounded above by $N_k= k |K|$, which is finite. Since there are finitely many isomorphism types of graphs on $N_k$ vertices, there are only finitely many possibilities for the structure of $X_i$, and so there is a subsequence \sseq{G}{A}\ of \seq{G}\ such that the $X_{i}$ are pairwise isomorphic for all $i \in A$.
	
	But since \seq{G}\ is Cauchy \wrt\ $d$, we can in fact choose the $X_i$ in such a way that for some $j_k\in \N$, the  $X_i$ are pairwise isomorphic for all $i>j_k$: indeed, for large enough $i,j$, if a graph $X$ on $N_k$ vertices is a RCIS of $G_i$ then it is also a RCIS of $G_j$ by the definition of $d$. We let $L_1$ be a rooted graph isomorphic to these $X_i$.
	
	We proceed to define $L_2, L_3, \ldots$ in a similar manner: we replace $k$ by $|L_1|$, and repeat the same construction to obtain $L_2$, and inductively, having defined $L_j$, we replace $k$ by $|L_j|$, and repeat the same construction to obtain $L_{j+1}$. Note that as $L_j$ is a RCIS of almost all $G_n$ by definition, there must be a RCIS of $L_{j+1}$ isomorphic to $L_j$. 
	
	Our limit graph $L$ will be constructed in such a way that every $L_j$ is a RCIS of $L$. For this, we label the vertices of each $L_j$ as $v_0, v_1, \ldots, v_{|L_j|}$ in such a way that $v_0$ is always the root, and the (induced) subgraph of $L_j$ spanned by $v_0, v_1, \ldots, v_{|L_{j-1}|}$ is isomorphic to $L_{j-1}$; this can be achieved by first fixing an isomorphism $f_j$ from $L_{j-1}$ to a RCIS of $L_j$, labelling each vertex in the image of $f_j$ by the same label as its preimages, and extending the labelling arbitrarily to the remaining vertices of $L_j$.
	
	We can now define $L= \bigcup L_j$ to be the graph on $v_0, v_1, \ldots$ whose-edge set is the union of the edge-sets of the $L_j$. We claim that $\seq{G}$ converges  to $L$. To see this, recall  that almost all $G_n$ lie in the same clopen ball of $(\CGc,d)$ of radius $1/k$ for every $k\in \N$, and we constructed $L$ in such a way that it also lies in that ball: indeed, any $k$-vertex graph that appears as an RCIS of infinitely many $G_n$ is a subgraph of some $L_j$, and therefore of $L$, and conversely, every  $k$-vertex RCIS of $L$ is contained is some $L_j$ and is therefore an RCIS of almost every $G_n$. This proves that $(\CGc,d)$ is complete.
\end{proof}

\begin{remark}
The limit graph $L$ is in general not uniquely determined by \seq{G}. It might depend on the choice of the isomorphisms $f_j$.
\end{remark}

Recall that a metric space is totally bounded, if for every $\eps>0$ it can be covered by finitely many balls of radius $\eps$.

\begin{lemma}\label{}	
	$(\CGc,d)$ is totally bounded.
\end{lemma}
\begin{proof}
	This follows immediately from \eqref{balls}: given \eps, we let $k=\ceil{1/\eps}$, and remark that $(\CGc,d)$ is covered by the at most $2^{2^{k \choose 2 }}$ ultra-balls of radius $1/k$.
\end{proof}

Since a metric space is compact \iff\ it is complete and totally bounded, we deduce

\begin{theorem}\label{compact}
	$(\CGc,d)$ is compact.
\end{theorem} 

\subsection{Convergence of random graphs}

Since $(\CGc,d)$ is a metric space, we can consider the weak topology on  the space $\mathcal M(\CGc)$ of probability distributions on $\CGc$. As  $(\CGc,d)$ is compact, so is  $\mathcal M(\CGc)$ by the standard theory. Thus every sequence of elements of $\mathcal M(\CGc)$ has a convergent subsequence with a limit in $\mathcal M(\CGc)$. Note that any probability distribution on $\CG$ that is concentrated on finite graphs can be thought of as an element of $\mathcal M(\CGc)$, therefore the last statement applies to sequences of such distributions too.

\begin{conjecture}\label{Croot}
	Let $\cs$ be a finite set of finite (connected) graphs, and $Ex^{\bullet}(\cs)$ the class of rooted (labelled or unlabelled) connected graphs with no minor in $\cs$. 
	Let $R_n$ be a uniformly random element of $Ex^{\bullet}(\cs)$ with $n$ vertices. Then $R_n$ converges weakly \wrt\ $d$.
\end{conjecture}

In the bounded-degree case (and in somewhat greater generality), if we start with an unrooted ---possibly random--- graph and choose a root uniformly at random, weak convergence translates to BS-convergence. We can apply this idea in greater generality using $d$:

\begin{conjecture}\label{Cunroot}
	Let $\cs$ be a finite set of finite (connected) graphs, and $Ex(\cs)$ the class of (labelled or unlabelled) connected graphs with no minor in $\cs$. Let $G_n$ be a uniformly random element of $Ex(\cs)$ with $n$ vertices, in which a root is chosen uniformly at random among all its vertices. Then $G_n$ converges weakly \wrt\ $d$. 
\end{conjecture} 

(For labelled graphs, this is equivalent to \Cnr{Croot}.)

\medskip
Perhaps an unsatisfactory aspect of the above conjectures is that the limit of a sequence of finite random graphs is not a random graph itself, but a random equivalence class of graphs under the $d=0$ relation. However, there is a way to translate this into a random graph: it is shown in \cite{GCcores} that for  every such equivalence class there is a uniquely determined graph that is an induced subgraph of all members of the class. This subgraph is defined as follows.

\newcommand{\gf}{ground floor}
\newcommand{\ff}{first floor}
\newcommand{\cor}{core}
Let $V_\infty(G)$ denote the set of infinite degree vertices of a graph $G$. The \defi{\gf} $G^0$ of a rooted graph $(G,o)$ is the component of $o$ in the subgraph $G - V_\infty(G)$. The \defi{\ff} $G^1$ of $G$ is the graph spanned by the vertices in $V_\infty(G)$ having a neighbour in $G^0$. The \defi{\cor} $C(G)$ of $(G,o)$ is the subgraph $G[G^0 \cup G^1]$ spanned by the vertices in the \gf\ and all their neighbours. 

It is proved in \cite{GCcores} that if $d(G,H)=0$, then the \cor s of $G$ and $H$ are isomorphic. Thus in our setup, we can consider the core of the limit equivalence class to be the random graph to which our sequence converges. Note that if $G$ is locally finite, then it is the unique element of its equivalence class and it coincides with its core; therefore, our notion of convergence generalises BS-convergence.

\bibliographystyle{plain}
\bibliography{collective}
\end{document}